\documentclass[12pt]{amsart}
\usepackage{amssymb,amsmath,amsthm,mathrsfs,multirow,xcolor,framed,url}
\usepackage[pdftex,
         pdfauthor={Dandan Chen, Rong Chen and Siyu Yin},
         pdftitle={On the total number of ones associated with cranks of partitions modulo 11},
         pdfsubject={MATHEMATICS},
         pdfkeywords={Total number of ones, Crank for partitions, Theta-functions.},
         pdfproducer={Latex with hyperref},
         pdfcreator={pdflatex}]{hyperref}
\usepackage [latin1]{inputenc}
\newtheorem{theorem}{Theorem}[section]
\newtheorem{lemma}[theorem]{Lemma}

\newtheorem{conj}[theorem]{Conjecture}

\theoremstyle{definition}
\newtheorem{definition}[theorem]{Definition}

\theoremstyle{remark}

\numberwithin{equation}{section}

\newcommand\nutwid{\overset {\text{\lower 3pt\hbox{$\sim$}}}\nu}

\allowdisplaybreaks

\newcommand\omycite[1]{}

\newcommand{\beqs}{\begin{equation*}}
\newcommand{\eeqs}{\end{equation*}}
\newcommand{\beq}{\begin{equation}}
\newcommand{\eeq}{\end{equation}}
\renewcommand{\MR}[1]{\href{http://www.ams.org/mathscinet-getitem?mr={#1}}{MR{#1}}}

\begin{document}
\title[Congruences for broken 3-diamond partitions]{A conjecture of Radu and Sellers on congruences modulo powers of 2 for broken 3-diamond partitions}


\author{Dandan Chen}
\address{Department of Mathematics, Shanghai University, People's Republic of China}
\address{Newtouch Center for Mathematics of Shanghai University, Shanghai, People's Republic of China}
\email{mathcdd@shu.edu.cn}
\author{Rong Chen}
\address{Department of Mathematics, Shanghai Normal University, Shanghai, China}
\email{rchen@shnu.edu.cn}
\author{Siyu Yin*}
\address{Department of Mathematics, Shanghai University, People's Republic of China}
\email{siyuyin@shu.edu.cn, siyuyin0113@126.com}


\subjclass[2010]{ 11P83, 05A17}

\date{}


\keywords{Broken $k$-diamonds; Congruences; Modular forms; Partitions }

\begin{abstract}
In 2007, Andrews and Paule introduced the family of functions $\Delta_k(n)$, which enumerate the number of broken $k$-diamond partitions for a fixed positive integer $k$. In 2013, Radu and Sellers completely characterized the parity of $\Delta_3(8n+r)$ for certain values of $r$ and proposed a conjecture on congruences modulo powers of $2$ for broken $3$-diamond partitions. In this paper, we employ an unconventional $U$-sequence to resolve the revised conjecture put forward by Radu and Sellers.
\end{abstract}

\maketitle


\section{Introduction}
A partition of a positive integer $n$ is defined as a sequence of positive integers in non-increasing order that sum to $n$. Denoted by $p(n)$, it represents the number of such partitions of $n$. The following celebrated congruences were discovered by Ramanujan in \cite{Ramanujan-1919}:
\begin{align*}
p(5n+4) &\equiv 0 \pmod{5},\\
p(7n+5) &\equiv 0 \pmod{7},\\
p(11n+6) &\equiv 0 \pmod{11},
\end{align*}
for every non-negative integer $n$. In fact, for $\alpha \geq 1$,
\begin{align}
p(5^{\alpha}n + \delta_{5,\alpha}) &\equiv 0 \pmod{5^{\alpha}}, \label{5-congruence}\\
p(7^{\alpha}n + \delta_{7,\alpha}) &\equiv 0 \pmod{7^{\lfloor \frac{\alpha+2}{2} \rfloor}}, \label{7-congruence}\\
p(11^{\alpha}n + \delta_{11,\alpha}) &\equiv 0 \pmod{11^{\alpha}} \label{11-congruence},
\end{align}
where $\delta_{t,\alpha}$ is the multiplicative inverse of $24$ modulo $t^{\alpha}$. Congruences \eqref{5-congruence} and \eqref{7-congruence} were first proved by Watson \cite{Watson-1938} in 1938. For an elementary proof of \eqref{5-congruence}, see Hirschhorn and Hunt \cite{Hirschhorn-1981}; for an elementary proof of \eqref{7-congruence}, see Garvan \cite{Garvan-1984}. Congruence \eqref{11-congruence} was proved by Atkin \cite{Atkin-1967} in 1967.

In 2007, Andrews and Paule \cite{Andrews-Paule-2007} introduced the family of functions $\Delta_k(n)$, which enumerate the number of broken $k$-diamond partitions for a fixed positive integer $k$. These are constructed such that the generating functions of their counting sequences $\{\Delta_k(n)\}_{n \geq 0}$ are closely related to modular forms. Namely,
\begin{align*}
\sum_{n=0}^{\infty} \Delta_k(n) q^n
&= \prod_{n=1}^{\infty} \frac{(1-q^{2n})(1-q^{(2k+1)n})}{(1-q^n)^3 (1-q^{(4k+2)n})} \\
&= q^{(k+1)/12} \frac{\eta(2\tau) \eta((2k+1)\tau)}{\eta(\tau)^3 \eta((4k+2)\tau)}, \qquad k \geq 1,
\end{align*}
where we recall the Dedekind eta function
\begin{align*}
\eta(\tau) = q^{\frac{1}{24}} \prod_{n=1}^{\infty} (1 - q^n) \qquad (q := e^{2\pi i \tau}).
\end{align*}
For $k = 1$, Andrews and Paule proved that for all $n \geq 0$, $\Delta_1(2n+1) \equiv 0 \pmod{3}$. Hirschhorn and Sellers \cite{Hirschhorn-Sellers-2007} considered the cases $k = 1$ and $k = 2$ and obtained the following parity results: for all $n \geq 0$,
\begin{align*}
&\Delta_1(4n+2) \equiv 0 \pmod{2},\\
&\Delta_1(4n+3) \equiv 0 \pmod{2},\\
&\Delta_2(10n+2) \equiv 0 \pmod{2},\\
&\Delta_2(10n+6) \equiv 0 \pmod{2}.
\end{align*}
 Besides, Chan \cite{Chan-2008} provided a different proof of the parity results for $\Delta_2$ mentioned above, as well as a number of congruences modulo powers of $5$. Paule and Radu \cite{Paule-Radu-2010} also proved several congruences modulo $5$ for broken $2$-diamond partitions, and they also shared conjectures related to broken $3$-diamond partitions modulo $7$ and broken $5$-diamond partitions modulo $11$, two of which were proved by Xiong \cite{Xiong-2011}.

In 2013, Radu and Sellers \cite{Radu-Sellers-2013} completely characterized the parity of $\Delta_3(8n+r)$ for certain $r$ and proposed the following conjecture modulo powers of $2$.
\begin{conj}
\label{mconj}
Let
\begin{align*}
\lambda_{\alpha}= \left\{\begin{array}{ll}
			\frac{2^{\alpha+1}+1}{3}, ~&\text{if}~\alpha~\text{is even},\\
			\frac{2^{\alpha}+1}{3}, ~&\text{if}~\alpha~\text{is odd}.
			\end{array}\right.
\end{align*}
Then for all $\alpha \geq 1$ and $n \geq 0$,
\begin{align}\label{conj-1}
\Delta_3(\lambda_{\alpha})\Delta_3(2^{\alpha+2}n+\lambda_{\alpha+2})\equiv\Delta_3(\lambda_{\alpha+2})\Delta_3(2^{\alpha}n+\lambda_{\alpha})\pmod{2^{\alpha}}
\end{align}
and
\begin{align}\label{conj-2}
\Delta_3(\lambda_{\alpha})\equiv 1 \pmod{2}.
\end{align}
\end{conj}
They also proved the case $\alpha = 1$. Later, Xia \cite{Xia-2014} demonstrated that \eqref{conj-1} holds for $\alpha = 2$ and that \eqref{conj-2} holds for all $\alpha \geq 1$. However, \eqref{conj-1} is incorrect in general. For example, let $\alpha = 3$ and $n = 1$; we have
\begin{align*}
&\Delta_3(\lambda_{3}) = \Delta_3(3) = 19,\\
&\Delta_3(2^5 + \lambda_{5}) = \Delta_3(43) = 236491535,\\
&\Delta_3(\lambda_{5}) = \Delta_3(2^3 + \lambda_{3}) = \Delta_3(11) = 2653,\\
&\Delta_3(\lambda_{3})\Delta_3(2^5 + \lambda_{5}) - \Delta_3(\lambda_{5})\Delta_3(2^3 + \lambda_3) = 2^2 \cdot 1121575189 \not\equiv 0 \pmod{2^3},
\end{align*}
which indicates that \eqref{conj-1} does not hold for $\alpha = 3$. The correct version of this conjecture, which we prove, is given below in Theorem \ref{thm}.

In this paper, we prove this conjecture by following the method in \cite{Paule-Radu-2012}. We associate the family of congruences with a sequence of functions defined on the classical modular curve $X_0(28)$ (a curve of genus $2$ with cusp count $6$). Let
$$
F(q) := q^{\frac{1}{3}} \frac{\eta(2\tau)\eta(7\tau)}{\eta(\tau)^3\eta(14\tau)} = \sum_{n=0}^{\infty} \Delta_3(n) q^n.
$$
We note that we employ an unconventional $U$-sequence, since $F(q)/F(q^4)$ is not a modular function, and there is no $U$-sequence of functions defined on the classical modular curve $X_0(14)$ (a curve of genus $1$ with cusp count $4$) associated with the family of congruences.

The paper is organized as follows. In Section \ref{sec-pre} we introduce the necessary notations and definitions. In Section \ref{sec-lemma} we derive modular equations on $\Gamma_0(28)$. Section \ref{sec-main} is dedicated to the proof of the revised Conjecture \ref{mconj}.

\section{Preliminaries}\label{sec-pre}

In this paper, we use the following conventions: $\mathbb{N}^{\ast} = \{1,2,\dots\}$ denotes the set of positive integers. The complex upper half-plane is denoted by $\mathbb{H} := \{\tau \in \mathbb{C} : \Im(\tau) > 0\}$. We will also use the shorthand notation for the Dedekind eta function:
\begin{align*}
\eta_n(\tau) := \eta(n\tau), \qquad n \in \mathbb{Z}, \; \tau \in \mathbb{H}.
\end{align*}
Throughout this paper, for $x \in \mathbb{R}$ the symbol $\lfloor x \rfloor$ denotes the largest integer less than or equal to $x$, and $\lceil x \rceil$ denotes the smallest integer greater than or equal to $x$.

Let $f = \sum_{n \in \mathbb{Z}} a_n q^n$, $f \neq 0$, be such that $a_n = 0$ for almost all $n < 0$. We define the order of $f$ (with respect to $q$) as the smallest integer $N$ such that $a_N \neq 0$, and write $N = \operatorname{ord}_q(f)$. For instance, let $F = f \circ t = \sum_{n \in \mathbb{Z}} a_n t^n$ where $t = \sum_{n \geq 1} b_n q^n$; then the $t$-order of $F$ is defined to be the smallest integer $N$ such that $a_N \neq 0$, and we write $N = \operatorname{ord}_t(F)$.


\begin{definition}
For $f: \mathbb{H} \rightarrow \mathbb{C}$ and $m \in \mathbb{N}^{\ast}$, we define $U_m(f): \mathbb{H} \rightarrow \mathbb{C}$ by
\begin{align*}
U_m(f)(\tau) := \frac{1}{m} \sum_{\lambda=0}^{m-1} f\!\left(\frac{\tau+\lambda}{m}\right), \qquad \tau \in \mathbb{H}.
\end{align*}
\end{definition}
The operator $U_m$ is linear over $\mathbb{C}$; moreover, it is easy to verify that
\begin{align*}
U_{mn} = U_m \circ U_n = U_n \circ U_m, \qquad m, n \in \mathbb{N}^{\ast}.
\end{align*}
The operators $U_m$, introduced by Atkin and Lehner \cite{Atkin-Lehner-1970}, are closely related to Hecke operators. They typically arise in the context of partition congruences \cite{Andrews-1976}, mostly due to the following property: if
\begin{align*}
f(\tau) = \sum_{n=-\infty}^{\infty} f_n q^n \qquad (q = e^{2\pi i \tau}),
\end{align*}
then
\begin{align*}
U_m(f)(\tau) = \sum_{n=-\infty}^{\infty} f_{mn} q^n.
\end{align*}

\begin{definition}\label{def-A-B}
For $f: \mathbb{H} \rightarrow \mathbb{C}$ we define $U^{(0)}(f), U^{(1)}(f): \mathbb{H} \rightarrow \mathbb{C}$ by $U^{(0)}(f) := U_2(Af)$ and $U^{(1)}(f) := U_2(Bf)$, where
\[
A := \frac{\eta_8^4 \eta_2^2}{\eta_4^2 \eta_1^4}, \qquad
B := \frac{\eta_1^4 \eta_2^2 \eta_8^4}{\eta_4^{10}}.
\]
\end{definition}

\begin{definition}\label{r:defL}
We define the $U$-sequence $(L_{\alpha})_{\alpha \geq 0}$ by $L_0 := \frac{\eta_7 \eta_4^4 \eta_1}{\eta_{14} \eta_2^5}$ and for $\alpha \geq 1$:
\[
L_{2\alpha-1} := U^{(0)}(L_{2\alpha-2}), \qquad
L_{2\alpha} := U^{(1)}(L_{2\alpha-1}).
\]
\end{definition}

The proof of the following lemma is completely analogous to \cite[ p.~23 ]{Atkin-1967} and is therefore omitted.

\begin{lemma}\label{lem-L-alpha}
For $\alpha \in \mathbb{N}^{\ast}$, we have
\begin{align*}
L_{2\alpha-1} &= q \prod_{n=1}^{\infty} \frac{(1-q^{4n})^4 (1-q^{2n})^2}{(1-q^n)^4} \sum_{n=0}^{\infty} \Delta_3(2^{2\alpha-1}n + \lambda_{2\alpha-1}) q^n,\\[4pt]
L_{2\alpha} &= q \prod_{n=1}^{\infty} \frac{(1-q^{n})^4 (1-q^{4n})^4}{(1-q^{2n})^6} \sum_{n=0}^{\infty} \Delta_3(2^{2\alpha}n + \lambda_{2\alpha}) q^n.
\end{align*}
\end{lemma}

\begin{definition}\label{t-p-y}
Let $t, p, p_0, y_0, p_1$ be functions defined on $\mathbb{H}$ as follows:
\begin{align*}
t &:= \frac{\eta_{28} \eta_4}{\eta_7 \eta_1}, ~~~~~~~~~~~~~~~~~~~~~
p := \frac{1}{2} \left( 1 + \frac{7 \eta_{14}^7 \eta_4^5 \eta_1^6}{\eta_{28}^3 \eta_7^2 \eta_2^{13}} \right), \\[4pt]
p_0 &:= \frac{\eta_7 \eta_4^7}{\eta_{28} \eta_1^7}, ~~~~~~~~~~~~~~~~~~~~~~~~~~~~~~~~~~~~~~~~~~
y_0 := \frac{7 \eta_{14}^5 \eta_4^5 \eta_2}{\eta_{28}^3 \eta_1^8}, ~~~~~~~~~~~~~~~~~
p_1 := \frac{1}{4} (p + 2p t + t) + t^2.
\end{align*}
\end{definition}


\section{Modular equation}
\label{sec-lemma}

Our proof of Radu and Sellers' conjecture relies on the identities in the Appendix. All these identities can be proved using modular function theory.

Regarding the valence formula in modular forms, Garvan has written a MAPLE package called ETA, which can prove Group I, Group III, and Group IV in the Appendix automatically. See
\begin{align}
\label{r:eta}
https://qseries.org/fgarvan/qmaple/ETA/
\end{align}
A tutorial for this package can be found in \cite{gtutorial}.

The identities in Group II of the Appendix cannot be proved directly, since $p_1$ is not an eta-quotient. However, we can use the ETA MAPLE package to prove that
\begin{align}
\label{r:u1}
U^{(0)}(t^3) &= p_0(256t^7 + 640t^6 + 864t^5 + 656t^4 + 276t^3 + 64t^2 + 7t) \nonumber \\
&\quad - y_0(32t^5 + 48t^4 + 28t^3 + 8t^2 + t); \\[4pt]
\label{r:u2}
U^{(0)}(t^4) &= p_0(1024t^9 + 3072t^8 + 4864t^7 + 4672t^6 + 2832t^5 + 1104t^4 + 260t^3 + 28t^2) \nonumber \\
&\quad - y_0(128t^7 + 256t^6 + 224t^5 + 112t^4 + 32t^3 + 4t^2); \\[4pt]
\label{r:u3}
U^{(0)}(yt) &= -p_0(14t + 7) + y_0(2t + 1); \\[4pt]
\label{r:u4}
U^{(0)}(yt^2) &= -p_0(56t^3 + 84t^2 + 28t) + y_0(8t^2 + 4t); \\[4pt]
\label{r:u5}
U^{(0)}(yt^3) &= -p_0(224t^5 + 448t^4 + 308t^3 + 84t^2 + 7t) + y_0(32t^4 + 36t^3 + 12t^2 + t);
\end{align}
where $y := 2p - 1$ is an eta-quotient. Then we deduce the identities in Group II using \eqref{r:u1}--\eqref{r:u5} along with the identities in Group I, since
\[
p_1 = \frac{1}{8} + \frac{1}{2}t + t^2 + \frac{1}{8}y + \frac{1}{4}yt.
\]
\begin{definition}\label{def-aj-t}
With $t = t(\tau)$ as in Definition \ref{t-p-y}, we define:
\begin{align*}
a_0(t) = -2t^2 - t, \qquad a_1(t) = -4t^2 - 2t.
\end{align*}
We define $s: \{0,1\} \times \{1,2\} \rightarrow \mathbb{Z}$ to be the unique function satisfying
\begin{align}\label{def-aj}
a_j(t) = \sum_{l=1}^2 s(j,l) \, 2^{\left\lfloor \frac{2l + j - 1}{2} \right\rfloor} t^l.
\end{align}
\end{definition}

\begin{lemma}\label{t-aj}
For $\lambda = 0$ or $1$, let
\begin{align*}
t_{\lambda}(\tau) := t\!\left(\frac{\tau + \lambda}{2}\right), \qquad \tau \in \mathbb{H}.
\end{align*}
Then, in the polynomial ring $\mathbb{C}(t)[X]$,
\begin{align*}
X^2 + a_1(t) X + a_0(t) = (X - t_0)(X - t_1).
\end{align*}
\end{lemma}

\begin{proof}
First we prove
\begin{align*}
a_0(t) = -\frac{\eta_2^3 \eta_{14}^3}{\eta_1^3 \eta_7^3} = -2t^2 - t.
\end{align*}
We have
\begin{align*}
a_0(t) &= t_0 t_1 \\
&= -\prod_{\lambda=0}^1 q^{1/2} \prod_{n=1}^{\infty} \frac{(1 - q^{14n})(1 - q^{2n})}{(1 - (-1)^{\lambda n} q^{7n/2})(1 - (-1)^{\lambda n} q^{n/2})} \\
&= -q \prod_{n=1}^{\infty} \frac{(1 - q^{14n})^3 (1 - q^{2n})^3}{(1 - q^{7n})^3 (1 - q^{n})^3} = -2t^2 - t,
\end{align*}
where the last equality was verified by Garvan's MAPLE package ETA \eqref{r:eta}.

For $a_1(t)$,
\begin{align*}
a_1(t) = -t_0 - t_1 
= -2 U_2(t) 
= -4t^2 - 2t,
\end{align*}
where the last equality was again verified by the ETA MAPLE package.
\end{proof}

\begin{lemma}\label{U2-j-jl2}
For $u: \mathbb{H} \rightarrow \mathbb{C}$ and $j \in \mathbb{Z}$,
\begin{align*}
U_2(u t^j) = -a_0(t) U_2(u t^{j-2}) - a_1(t) U_2(u t^{j-1}).
\end{align*}
\end{lemma}

\begin{proof}
For $\lambda \in \{0,1\}$, Lemma \ref{t-aj} implies
\begin{align*}
t_{\lambda}^2 + a_1(t) t_{\lambda} + a_0(t) = 0.
\end{align*}
Multiplying both sides by $u_{\lambda} t_{\lambda}^{j-2}$, where $u_{\lambda}(\tau) := u\!\left(\frac{\tau + \lambda}{2}\right)$, gives
\begin{align*}
u_{\lambda} t_{\lambda}^j + a_1(t) u_{\lambda} t_{\lambda}^{j-1} + a_0(t) u_{\lambda} t_{\lambda}^{j-2} = 0.
\end{align*}
Summing both sides over all $\lambda \in \{0,1\}$ completes the proof of the lemma.
\end{proof}

According to Lemma \ref{U2-j-jl2}, we can calculate $U^{(0)}(t^k)$, $U^{(0)}(p_1 t^k)$, $U^{(1)}(y_0 t^k)$, and $U^{(1)}(p_0 t^k)$ using the identities in the Appendix for all $k \geq 1$ recursively. For example,
\begin{align}
\label{r:y0t3}
U^{(1)}(y_0 t^3) &= (4t^2 + 2t) U^{(1)}(y_0 t^2) + (2t^2 + t) U^{(1)}(y_0 t) \nonumber \\
&= p_1 \bigl( 2^{17} t^{13} + 2^{19} t^{12} + 125 \cdot 2^{13} t^{11} + 315 \cdot 2^{12} t^{10} + 2261 \cdot 2^9 t^9 + 1505 \cdot 2^9 t^8 \nonumber \\
&\qquad + 6031 \cdot 2^6 t^7 + 1131 \cdot 2^7 t^6 + 4953 \cdot 2^3 t^5 + 1875 \cdot 2^2 t^4 + 441 \cdot 2 t^3 + 49 t^2 \bigr) \nonumber \\
&\quad - \bigl( 2^{17} t^{15} + 9 \cdot 2^{16} t^{14} + 145 \cdot 2^{13} t^{13} + 93 \cdot 2^{14} t^{12} + 2693 \cdot 2^9 t^{11} + 3591 \cdot 2^8 t^{10} \nonumber \\
&\qquad + 7165 \cdot 2^6 t^9 + 5305 \cdot 2^5 t^8 + 5675 \cdot 2^3 t^7 + 1035 \cdot 2^3 t^6 + 463 \cdot 2 t^5 + 13 \cdot 2^2 t^4 \bigr).
\end{align}

\section{Radu and Sellers' conjecture}
\label{sec-main}

\begin{lemma}\label{order}
Let $v_1, v_2, u: \mathbb{H} \rightarrow \mathbb{C}$ and $l \in \mathbb{Z}$. Suppose that for $k = l$ or $k = l+1$ there exist polynomials $p_k^{(1)}(t), p_k^{(2)}(t) \in \mathbb{Z}[t]$ such that
\begin{align}\label{ord-1}
U_2(u t^k) = v_1 p_k^{(1)}(t) + v_2 p_k^{(2)}(t)
\end{align}
and
\begin{align}\label{ord-2}
\operatorname{ord}_t\!\left(p_k^{(i)}(t)\right) \geq \left\lceil \frac{k + s_i}{2} \right\rceil, \qquad i \in \{1,2\},
\end{align}
for some fixed integers $s_1$ and $s_2$. Then there exist families of polynomials $p_k^{(1)}(t), p_k^{(2)}(t) \in \mathbb{Z}[t]$ for all integers $k \geq l$ such that \eqref{ord-1} and \eqref{ord-2} hold.
\end{lemma}
\begin{proof}
Let $N > l+1$ be an integer and assume by induction that there exist families of polynomials $p_k^{(i)}(t)$, $i \in \{1,2\}$, such that \eqref{ord-1} and \eqref{ord-2} hold for $l \leq k \leq N-1$. Suppose
\begin{align*}
p_k^{(i)}(t) = \sum_{n \geq \left\lceil \frac{k + s_i}{2} \right\rceil} c_i(k,n) \, t^n, \qquad l \leq k \leq N-1,
\end{align*}
with integers $c_i(k,n)$. Applying Lemma \ref{U2-j-jl2} we obtain:
\begin{align*}
U_2(u t^N) &= -\sum_{j=0}^1 a_j(t) \, U_2(u t^{N+j-2}) \\
&= -\sum_{j=0}^1 a_j(t) \sum_{i=1}^2 v_i \sum_{n \geq \left\lceil \frac{N+j-2+s_i}{2} \right\rceil} c_i(N+j-2,n) \, t^n \\
&= -\sum_{i=1}^2 v_i \sum_{j=0}^1 a_j(t) t^{-1} \sum_{n \geq \left\lceil \frac{N+j+s_i}{2} \right\rceil} c_i(N+j-2,n-1) \, t^n.
\end{align*}
Recalling the fact that $a_j(t) t^{-1}$ for $0 \leq j \leq 1$ is a polynomial in $t$, this determines polynomials $p_N^{(i)}(t)$ with the desired properties.
\end{proof}

\begin{lemma}\label{U-power}
Let $v_1, v_2, u: \mathbb{H} \rightarrow \mathbb{C}$ and $l \in \mathbb{Z}$. Suppose that for $k = l$ or $k = l+1$ there exist polynomials $p_k^{(i)}(t) \in \mathbb{Z}[t]$, $i \in \{1,2\}$, such that
\begin{align}\label{power-1}
U_2(u t^k) = v_1 p_k^{(1)}(t) + v_2 p_k^{(2)}(t)
\end{align}
where
\begin{align}\label{power-2}
p_k^{(i)}(t) = \sum_n c_i(k,n) \, 2^{\left\lfloor \frac{2n - k + \gamma_i}{2} \right\rfloor} t^n,
\end{align}
with integers $\gamma_i$ and $c_i(k,n)$. Then there exist families of polynomials $p_k^{(i)}(t) \in \mathbb{Z}[t]$ for all integers $k \geq l$, of the form \eqref{power-2}, for which property \eqref{power-1} holds.
\end{lemma}

\begin{proof}
Suppose that for an integer $N > l+1$ there exist families of polynomials $p_k^{(i)}(t)$, $i \in \{1,2\}$, of the form \eqref{power-2} satisfying property \eqref{power-1} for $l \leq k \leq N-1$. We proceed by mathematical induction on $N$. Applying Lemma \ref{U2-j-jl2} and using the induction base \eqref{power-1} and \eqref{power-2} we obtain:
\begin{align*}
U_2(u t^N) = -\sum_{j=0}^1 a_j(t) \sum_{i=1}^2 v_i \sum_n c_i(N+j-2,n) \, 2^{\left\lfloor \frac{2n - (N+j-2) + \gamma_i}{2} \right\rfloor} t^n.
\end{align*}
Utilizing \eqref{def-aj} this rewrites as:
\begin{align*}
U_2(u t^N) &= -\sum_{j=0}^1 \sum_{l=1}^2 s(j,l) \, 2^{\left\lfloor \frac{2l + j - 1}{2} \right\rfloor} t^l \sum_{i=1}^2 v_i \sum_n c_i(N+j-2,n) \, 2^{\left\lfloor \frac{2n - (N+j-2) + \gamma_i}{2} \right\rfloor} t^n \\
&= -\sum_{i=1}^2 v_i \sum_{j=0}^1 \sum_{l=1}^2 \sum_n s(j,l) \, c_i(N+j-2,n-l) \, 2^{\left\lfloor \frac{2(n-l) - (N+j-2) + \gamma_i}{2} \right\rfloor + \left\lfloor \frac{2l + j - 1}{2} \right\rfloor} t^n.
\end{align*}
The induction step is completed by simplifying the exponent of $2$ as follows:
\begin{align*}
&\left\lfloor \frac{2(n-l) - (N+j-2) + \gamma_i}{2} \right\rfloor + \left\lfloor \frac{2l + j - 1}{2} \right\rfloor \\
&\qquad \geq \left\lfloor \frac{2(n-l) - (N+j-2) + \gamma_i + 2l + j - 2}{2} \right\rfloor \\
&\qquad = \left\lfloor \frac{2n - N + \gamma_i}{2} \right\rfloor.
\end{align*}
Thus for $N \geq l$, $U_2(u t^N)$ has the desired expression.
\end{proof}

\begin{definition}
A map $a: \mathbb{Z} \rightarrow \mathbb{Z}$ is called a discrete function if it has finite support. A map $a: \mathbb{Z} \times \mathbb{Z} \rightarrow \mathbb{Z}$ is called a discrete array if for each $i \in \mathbb{Z}$ the map $a(i, -): \mathbb{Z} \rightarrow \mathbb{Z}$, $j \mapsto a(i,j)$, has finite support.
\end{definition}

\begin{lemma}\label{U-fundamental}
Given $U^{(0)}(f)$ and $U^{(1)}(f)$ as in Definition \ref{def-A-B}, and $p_0, p_1, y_0$ as in Definition \ref{t-p-y}. Then there exist discrete arrays $a_i$, $b_i$, $c_i$, and $d_i$, with $i \in \{0,1\}$, such that the following relations hold for all $k \in \mathbb{N}^{\ast}$:
\begin{align}
&U^{(0)}(t^k) = p_0 \sum_{n \geq \lceil \frac{k-1}{2} \rceil} a_0(k,n) \, 2^{\left\lfloor \frac{2n - k + 2}{2} \right\rfloor} t^n + y_0 \sum_{n \geq \lceil \frac{k-1}{2} \rceil} a_1(k,n) \, 2^{\left\lfloor \frac{2n - k + 2}{2} \right\rfloor} t^n, \label{U-0-t} \\[4pt]
&U^{(0)}(p_1 t^k) = p_0 \sum_{n \geq \lceil \frac{k+1}{2} \rceil} b_0(k,n) \, 2^{\left\lfloor \frac{2n - k + 1}{2} \right\rfloor} t^n + y_0 \sum_{n \geq \lceil \frac{k+1}{2} \rceil} b_1(k,n) \, 2^{\left\lfloor \frac{2n - k - 1}{2} \right\rfloor} t^n, \label{U-0-p1t} \\[4pt]
&U^{(1)}(p_0 t^k) = p_1 \sum_{n \geq \lceil \frac{k}{2} \rceil} c_0(k,n) \, 2^{\left\lfloor \frac{2n - k + 1}{2} \right\rfloor} t^n + \sum_{n \geq \lceil \frac{k+4}{2} \rceil} c_1(k,n) \, 2^{\left\lfloor \frac{2n - k - 3}{2} \right\rfloor} t^n, \label{U-1-p0t} \\[4pt]
&U^{(1)}(y_0 t^k) = p_1 \sum_{n \geq \lceil \frac{k}{2} \rceil} d_0(k,n) \, 2^{\left\lfloor \frac{2n - k}{2} \right\rfloor} t^n + \sum_{n \geq \lceil \frac{k+4}{2} \rceil} d_1(k,n) \, 2^{\left\lfloor \frac{2n - k - 4}{2} \right\rfloor} t^n. \label{U-1-y0t}
\end{align}
\end{lemma}
\begin{proof}
The Appendix lists eight fundamental relations. The two fundamental relations of Group I fit the pattern of relation \eqref{U-0-t} for two values of $k$. The same observation applies to Groups II, III, and IV with regard to relations \eqref{U-0-p1t}, \eqref{U-1-y0t}, and \eqref{U-1-p0t}, respectively. Applying Lemmas \ref{order} and \ref{U-power}, we immediately prove the statement for all $k \in \mathbb{N}^{\ast}$.
\end{proof}

We verify that
\begin{align}
\label{r:L1id}
L_1 = U^{(0)}(L_0) = -4p_0 t + y_0 t,
\end{align}
using Garvan's MAPLE package ETA \eqref{r:eta}. We then note that $p$ and $p_1$ have expansions in powers of $q$ with coefficients in $\mathbb{Z}$.

Since $(1 - q^k)^2 \equiv 1 - q^{2k} \pmod{2}$, we have $\eta_k^2 \equiv \eta_{2k} \pmod{2}$. Then
\begin{align*}
2p = 1 + 7 \frac{\eta_{14}^7 \eta_4^5 \eta_1^6}{\eta_{28}^3 \eta_7^2 \eta_2^{13}}
\equiv 1 + 7 \frac{\eta_{14}^7 \eta_4^5 \eta_2^3}{\eta_{28}^3 \eta_{14} \eta_2^{13}}
\equiv 1 + 7 \frac{\eta_{14}^6 \eta_4^5}{\eta_{28}^3 \eta_2^{10}}
\equiv 8 \equiv 0 \pmod{2},
\end{align*}
which implies that the function $p$ has expansions in powers of $q$ with coefficients in $\mathbb{Z}$.

We now have that the coefficients of $4p_1 = p + 2p t + t + 4t^2$ are integral. Hence, by the relations for $U^{(1)}(p_0 t^2)$ given in the Appendix,
\[
U^{(1)}(p_0 t^2) = p_1 t + 4p_1 f(t) + g(t),
\]
where $f(t)$ and $g(t)$ are polynomials in $t$ with integral coefficients, we see that $p_1 t$ has expansions in powers of $q$ with coefficients in $\mathbb{Z}$ as does $p_1$. Furthermore, using \eqref{r:L1id} and combining the relations for $U^{(1)}(p_0 t)$, $U^{(1)}(y_0 t)$, $U^{(0)}(p_1 t)$ given in the Appendix together with $U^{(0)}(t^4)$ given in \eqref{r:u2}, we now state the following representation.

\begin{lemma}\label{r:L123lm}
Given $L_\alpha$ in Definition \ref{r:defL}, we have
\begin{align}
L_2 &= 3p_1 t - 2t^4 + 4f(p_1, t), \label{eq-L2} \\
L_3 &= y_0 t + 2y_0 t^2 + 4g(p_0, y_0, t), \nonumber
\end{align}
for some polynomials $f$ and $g$ with integral coefficients.
\end{lemma}

\begin{lemma}\label{L-alpha}
For $\alpha \geq 1$ there exist integers $d_{n,i}^{(\alpha)}$, $i = 0, 1, 2, 3$, such that
\begin{align}\label{L-odd}
L_{2\alpha-1} = p_0 \sum_{k=1}^{\infty} d_{k,0}^{(2\alpha-1)} \, 2^{k-1} t^k + y_0 \sum_{k=1}^{\infty} d_{k,2}^{(2\alpha-1)} \, 2^{k-1} t^k,
\end{align}
and
\begin{align}\label{L-even}
L_{2\alpha} = p_1 \sum_{k=1}^{\infty} d_{k,1}^{(2\alpha)} \, 2^{k-1} t^k + \sum_{k=3}^{\infty} d_{k,3}^{(2\alpha)} \, 2^{k-3} t^k.
\end{align}
\end{lemma}

\begin{proof}
By \eqref{r:L1id} we see that $L_1 = -4p_0 t + y_0 t$ has the form given in \eqref{L-odd}. Assume that $L_{2\alpha-1}$ has the form given in \eqref{L-odd} for a fixed $\alpha$. Then
\begin{align}\label{even-odd-1}
L_{2\alpha} = U^{(1)}(L_{2\alpha-1}) = \sum_{k=1}^{\infty} d_{k,0}^{(2\alpha-1)} \, 2^{k-1} U^{(1)}(p_0 t^k) + \sum_{k=1}^{\infty} d_{k,2}^{(2\alpha-1)} \, 2^{k-1} U^{(1)}(y_0 t^k).
\end{align}
We aim to show that each sum on the right-hand side of \eqref{even-odd-1} satisfies the form in \eqref{L-even}. From \eqref{U-1-p0t} and \eqref{U-1-y0t},
\begin{align*}
L_{2\alpha} &= p_1 \sum_{k=1}^{\infty} \sum_{n \geq \lceil \frac{k}{2} \rceil} d_{k,0}^{(2\alpha-1)} c_0(k,n) \, 2^{k-1 + \left\lfloor \frac{2n - k + 1}{2} \right\rfloor} t^n \\
&\quad + \sum_{k=1}^{\infty} \sum_{n \geq \lceil \frac{k+4}{2} \rceil} d_{k,0}^{(2\alpha-1)} c_1(k,n) \, 2^{k-1 + \left\lfloor \frac{2n - k - 3}{2} \right\rfloor} t^n \\
&\quad + p_1 \sum_{k=1}^{\infty} \sum_{n \geq \lceil \frac{k}{2} \rceil} d_{k,2}^{(2\alpha-1)} d_0(k,n) \, 2^{k-1 + \left\lfloor \frac{2n - k}{2} \right\rfloor} t^n \\
&\quad + \sum_{k=1}^{\infty} \sum_{n \geq \lceil \frac{k+4}{2} \rceil} d_{k,2}^{(2\alpha-1)} d_1(k,n) \, 2^{k-1 + \left\lfloor \frac{2n - k - 4}{2} \right\rfloor} t^n.
\end{align*}
For $k \geq 1$ and $n \geq \lceil \frac{k}{2} \rceil$ we have
\begin{align*}
k-1 + \left\lfloor \frac{2n - k + 1}{2} \right\rfloor &= \left\lfloor \frac{2n - k + 1 + 2k - 2}{2} \right\rfloor = n + \left\lfloor \frac{k-1}{2} \right\rfloor \geq n, \\
k-1 + \left\lfloor \frac{2n - k}{2} \right\rfloor &= \left\lfloor \frac{2n - k + 2k - 2}{2} \right\rfloor = n + \left\lfloor \frac{k-2}{2} \right\rfloor \geq n-1.
\end{align*}
For $k \geq 1$ and $n \geq \lceil \frac{k}{2} \rceil + 2$ we have
\begin{align*}
k-1 + \left\lfloor \frac{2n - k - 3}{2} \right\rfloor &= \left\lfloor \frac{2n - k - 3 + 2k - 2}{2} \right\rfloor = n + \left\lfloor \frac{k-5}{2} \right\rfloor \geq n-2, \\
k-1 + \left\lfloor \frac{2n - k - 4}{2} \right\rfloor &= \left\lfloor \frac{2n - k - 4 + 2k - 2}{2} \right\rfloor = n + \left\lfloor \frac{k-6}{2} \right\rfloor \geq n-3.
\end{align*}
Thus the sums on the right-hand side of \eqref{even-odd-1} have the form of \eqref{L-even}. Hence $L_{2\alpha}$ has the desired form. The proof that the correct form of $L_{2\alpha}$ implies the correct form of $L_{2\alpha+1}$ is analogous. The general result follows by induction.
\end{proof}

By Lemma \ref{U-power} and Lemma \ref{L-alpha} we set
\begin{align*}
L_{2\alpha-1} &= p_0 \sum_{n=1}^{\infty} l_{n,0}^{(2\alpha-1)} t^n + y_0 \sum_{n=1}^{\infty} l_{n,2}^{(2\alpha-1)} t^n, \\[4pt]
L_{2\alpha} &= p_1 \sum_{n=1}^{\infty} l_{n,1}^{(2\alpha)} t^n + \sum_{n=3}^{\infty} l_{n,3}^{(2\alpha)} t^n.
\end{align*}
Next we define
\begin{align*}
D^{(\alpha)}(l_{m,i}, l_{n,j}) := l_{m,i}^{(\alpha)} l_{n,j}^{(\alpha+2)} - l_{m,i}^{(\alpha+2)} l_{n,j}^{(\alpha)},
\end{align*}
and denote by $\pi(n)$ the $2$-adic order of $n$ (i.e., the highest power of $2$ that divides $n$). By Lemma \ref{L-alpha}, we note that for $m, n \in \mathbb{N}$ and $\alpha \geq 1$, the coefficients $l_{n,1}$ and $l_{m,3}$ appear only in $L_{2\alpha}$, while $l_{n,0}$ and $l_{m,2}$ appear only in $L_{2\alpha-1}$.
\begin{lemma}\label{pi-L}
For $\alpha \geq 1$ and $i,j \in \{0,1,2,3\}$, the following inequalities hold:
\begin{itemize}
\item[(1)] For the odd cases we have
\begin{align}
&\pi\bigl(D^{(2\alpha-1)}(l_{m,i},l_{n,j})\bigr) \geq 2\alpha - 2 + m + n + \lambda_{i,j}, \label{pi-odd}\\
&\pi\bigl(D^{(2\alpha-1)}(l_{1,2},l_{2,0})\bigr) \geq 2\alpha, \label{pi-odd-1220}\\
&\pi\bigl(D^{(2\alpha-1)}(l_{1,2},2l_{2,2} + l_{3,2})\bigr) \geq 2\alpha; \label{pi-odd-123}
\end{align}

\item[(2)] For the even cases we have
\begin{align}
&\pi\bigl(D^{(2\alpha)}(l_{m,i},l_{n,j})\bigr) \geq 2\alpha - 1 + m + n + \lambda_{i,j} - \delta(m,i,n,j), \label{pi-even}\\
&\pi\bigl(D^{(2\alpha)}(l_{1,1},l_{3,3})\bigr) \geq 2\alpha, \label{pi-even-1133}\\
&\pi\bigl(D^{(2\alpha)}(l_{2,1},l_{3,3})\bigr) \geq 2\alpha + 1; \label{pi-even-2133}
\end{align}
\end{itemize}
where
\begin{align*}
\lambda_{i,j} = \left\{\begin{array}{ll}
            -2, & \text{if } (i,j) = (1,1), (0,2), (2,0), \text{ or } (0,0), \\
            -3, & \text{if } (i,j) = (2,2), \\
            -4, & \text{if } (i,j) = (1,3) \text{ or } (3,1), \\
            -6, & \text{if } (i,j) = (3,3),
        \end{array}\right.
\end{align*}
\begin{align*}
\delta(m,i,n,j) = \left\{\begin{array}{ll}
            1, & \text{if } (m,i,n,j) \in G \text{ or } (n,j,m,i) \in G, \\
            0, & \text{otherwise},
        \end{array}\right.
\end{align*}
and
\[
G := \bigl\{ (m,i,n,j) : (m,i) \in \{(5,1), (6,3), (7,3)\} \text{ and } (n,j) \in \{(4,3), (5,3)\} \bigr\}.
\]
\end{lemma}
\begin{proof}
We prove this lemma by mathematical induction. First, we consider the case $\alpha = 1$. Since $l_{n,j}^{(1)} = 0$ except $l_{1,0}^{(1)} = -4$ and $l_{1,2}^{(1)} = 1$, from equation \eqref{L-odd} we have
\[
\pi\bigl(D^{(1)}(l_{m,i}, l_{n,j})\bigr) = \infty
\]
when $m, n > 1$,
\[
\pi\bigl(D^{(1)}(l_{1,i}, l_{n,j})\bigr) = \pi\bigl(-D^{(1)}(l_{n,j}, l_{1,i})\bigr) = \pi\bigl(l_{1,i}^{(1)} l_{n,j}^{(3)}\bigr) \geq n + 1 + \lambda_{i,j}
\]
when $n > 1$, and
\[
\pi\bigl(D^{(1)}(l_{1,0}, l_{1,2})\bigr) \geq 0 = 2 + \lambda_{0,2},
\]
which proves \eqref{pi-odd} for $\alpha = 1$. In particular, we obtain that
\begin{align*}
&\pi\bigl(D^{(1)}(l_{1,2}, l_{2,0})\bigr) = \pi\bigl(l_{2,0}^{(3)}\bigr) \geq 2,\\
&\pi\bigl(D^{(1)}(l_{1,2}, 2l_{2,2} + l_{3,2})\bigr) = \pi\bigl(2l_{2,2}^{(3)} + l_{3,2}^{(3)}\bigr) \geq 2,
\end{align*}
by Lemma \ref{r:L123lm}, which proves that \eqref{pi-odd-1220} and \eqref{pi-odd-123} hold for $\alpha = 1$.

Second, we assume that \eqref{pi-odd}--\eqref{pi-odd-123} hold for $\alpha = s$ and prove that \eqref{pi-even}--\eqref{pi-even-2133} hold for $\alpha = s$. We now compare the coefficients on both sides of
\begin{align*}
L_{2s} = p_1 \sum_{n=1}^{\infty} l_{n,1}^{(2s)} t^n + \sum_{n=3}^{\infty} l_{n,3}^{(2s)} t^n = \sum_{k=1}^{\infty} l_{k,0}^{(2s-1)} U^{(1)}(p_0 t^k) + \sum_{k=1}^{\infty} l_{k,2}^{(2s-1)} U^{(1)}(y_0 t^k).
\end{align*}
Denote $p_2 = y_0$ and $p_3 = 1$. For $u = 0$ or $2$, define $x_{k,u}(m,i) := [p_i t^m] U^{(1)}(p_u t^k)$ to be the coefficient of $p_i t^m$ in $U^{(1)}(p_u t^k)$, namely
\begin{align*}
U^{(1)}(p_u t^k) = p_1 \sum_n x_{k,u}(n,1) t^n + \sum_n x_{k,u}(n,3) t^n.
\end{align*}
Then we derive that $l_{n,i}^{(2s)} = \sum_{k,u} x_{k,u}(n,i) l_{k,u}^{(2s-1)}$ and
\begin{align}\label{2s-re}
D^{(2s)}(l_{m,i}, l_{n,j}) = \sum_{k,u,r,v} x_{k,u}(m,i) x_{r,v}(n,j) D^{(2s-1)}(l_{k,u}, l_{r,v}).
\end{align}
By \eqref{U-1-p0t} and \eqref{U-1-y0t}, we obtain $\pi\bigl(x_{k,u}(m,i)\bigr) \geq \left\lfloor \frac{2m - k + \gamma_{u,i}}{2} \right\rfloor$, where $\gamma_{0,1} = 1$, $\gamma_{2,1} = 0$, $\gamma_{0,3} = -3$, and $\gamma_{2,3} = -4$. Now we first consider $(k,u,r,v) \in \{(1,2,2,2), (1,2,3,2), (3,2,1,2), (2,2,1,2)\}$. After checking all coefficients of $U^{(1)}(y_0 t)$ and $U^{(1)}(y_0 t^2)$ given in the Appendix and $U^{(1)}(y_0 t^3)$ in \eqref{r:y0t3}, we obtain that
\begin{align}
\label{r:even1}
&\pi\!\left( \sum_{r=2,3} \bigl( x_{1,2}(m,i) x_{r,2}(n,j) + x_{r,2}(m,i) x_{1,2}(n,j) \bigr) D^{(2s-1)}(l_{1,2}, l_{r,2}) \right) \nonumber \\
&\qquad \geq 2s - 1 + m + n + \lambda_{i,j} - \delta(m,i,n,j),
\end{align}
which is based on inequalities \eqref{pi-odd} and \eqref{pi-odd-123} for $\alpha = s$.

Apart from $(k,u,r,v) \in \{(1,2,2,2), (1,2,3,2), (3,2,1,2), (2,2,1,2)\}$, we have
\begin{align}
\label{r:even}
&\pi\!\left( \sum_{k,u,r,v} x_{k,u}(m,i) x_{r,v}(n,j) D^{(2s-1)}(l_{k,u}, l_{r,v}) \right) \nonumber \\
&\qquad \geq \min_{(k,u) \neq (r,v)} \left( k + r + 2s - 2 + \lambda_{u,v} + \left\lfloor \frac{2m - k + \gamma_{u,i}}{2} \right\rfloor + \left\lfloor \frac{2n - r + \gamma_{v,j}}{2} \right\rfloor \right) \nonumber \\
&\qquad = \min_{(k,u) \neq (r,v)} \left( m + n + 2s - 2 + \lambda_{u,v} + \left\lfloor \frac{k + \gamma_{u,i}}{2} \right\rfloor + \left\lfloor \frac{r + \gamma_{v,j}}{2} \right\rfloor \right) \nonumber \\
&\qquad \geq 2s - 1 + m + n + \lambda_{i,j}.
\end{align}
By \eqref{2s-re}--\eqref{r:even} we obtain that inequality \eqref{pi-even} holds for $\alpha = s$. For $\pi\bigl(D^{(2s)}(l_{1,1}, l_{3,3})\bigr)$ and $\pi\bigl(D^{(2s)}(l_{2,1}, l_{3,3})\bigr)$, based on Lemma \ref{U2-j-jl2} we derive:
\begin{align}
\label{r:11}
l_{1,1}^{(2s)} &= 4 l_{1,0}^{(2s-1)} + l_{2,0}^{(2s-1)} + 35 l_{1,2}^{(2s-1)} + 7 l_{2,2}^{(2s-1)}, \\[4pt]
\label{r:21}
l_{2,1}^{(2s)} &= 40 l_{1,0}^{(2s-1)} + 20 l_{2,0}^{(2s-1)} + 6 l_{3,0}^{(2s-1)} + l_{4,0}^{(2s-1)} \nonumber \\
&\quad + 420 l_{1,2}^{(2s-1)} + 182 l_{2,2}^{(2s-1)} + 49 l_{3,2}^{(2s-1)} + 7 l_{4,2}^{(2s-1)}, \\[4pt]
\label{r:33}
l_{3,3}^{(2s)} &= -4 l_{1,0}^{(2s-1)} - l_{2,0}^{(2s-1)} - 36 l_{1,2}^{(2s-1)} - 8 l_{2,2}^{(2s-1)}.
\end{align}
Rewriting $\pi\bigl(D^{(2s)}(l_{1,1}, l_{3,3})\bigr)$ as $\pi\bigl(D^{(2s-1)}(4l_{1,0} + l_{2,0} + 35l_{1,2} + 7l_{2,2},\; -4l_{1,0} - l_{2,0} - 36l_{1,2} - 8l_{2,2})\bigr)$, we obtain $\pi\bigl(D^{(2s)}(l_{1,1}, l_{3,3})\bigr) \geq 2s$ by \eqref{pi-odd} and \eqref{pi-odd-1220}. Similarly, we obtain $\pi\bigl(D^{(2s)}(l_{2,1}, l_{3,3})\bigr) \geq 2s + 1$ by \eqref{pi-odd}.

Finally, we assume that \eqref{pi-odd}--\eqref{pi-even-2133} hold for $\alpha = s$ and prove that \eqref{pi-odd}--\eqref{pi-odd-123} hold for $\alpha = s+1$. For $u = 1$ or $3$, define $y_{k,u}(m,i) := [p_i t^m] U^{(0)}(p_u t^k)$ to be the coefficient of $p_i t^m$ in $U^{(0)}(p_u t^k)$, namely
\begin{align*}
U^{(0)}(p_u t^k) = p_0 \sum_n y_{k,u}(n,0) t^n + y_0 \sum_n y_{k,u}(n,2) t^n.
\end{align*}
Then
\begin{align}\label{r:2s-re}
D^{(2s+1)}(l_{m,i}, l_{n,j}) = \sum_{k,u,r,v} y_{k,u}(m,i) y_{r,v}(n,j) D^{(2s)}(l_{k,u}, l_{r,v}).
\end{align}
With the help of \eqref{U-0-t} and \eqref{U-0-p1t}, we have $\pi\bigl(y_{k,u}(m,i)\bigr) \geq \left\lfloor \frac{2m - k + \gamma_{u,i}}{2} \right\rfloor$, where $\gamma_{1,0} = 1$, $\gamma_{1,2} = -1$, $\gamma_{3,0} = 2$, and $\gamma_{3,2} = 2$. Apart from $(k,u,r,v) \in \{(1,1,3,3), (2,1,3,3), (3,3,1,1), (3,3,2,1)\}$, we obtain that
\begin{align}
\label{r:odd}
&\pi\!\left( \sum_{k,u,r,v} y_{k,u}(m,i) y_{r,v}(n,j) D^{(2s)}(l_{k,u}, l_{r,v}) \right) \nonumber \\
&\qquad \geq \min_{(k,u) \neq (r,v)} \left( m + n + 2s - 1 + \lambda_{u,v} - \delta(k,u,r,v) + \left\lfloor \frac{k + \gamma_{u,i}}{2} \right\rfloor + \left\lfloor \frac{r + \gamma_{v,j}}{2} \right\rfloor \right) \nonumber \\
&\qquad \geq 2s + m + n + \lambda_{i,j},
\end{align}
where the last inequality is determined by \eqref{pi-even} for $\alpha = s$. By \eqref{pi-even-1133} and \eqref{pi-even-2133} we have
\begin{align}
\label{r:y1}
&\pi\!\left( y_{1,1}(m,i) y_{3,3}(n,j) D^{(2s)}(l_{1,1}, l_{3,3}) \right) \geq 2s + m + n - 2, \\[4pt]
\label{r:y2}
&\pi\!\left( y_{2,1}(m,i) y_{3,3}(n,j) D^{(2s)}(l_{2,1}, l_{3,3}) \right) \geq 2s + m + n - 2.
\end{align}
Hence \eqref{pi-odd} holds for $\alpha = s+1$ by \eqref{r:2s-re}--\eqref{r:y2}.

Next, we consider \eqref{pi-odd-1220} and \eqref{pi-odd-123}. By Lemma \ref{U2-j-jl2}, we obtain
\begin{align}
\label{r:l12}
l_{1,2}^{(2s+1)} &= -l_{3,3}^{(2s)} - l_{1,1}^{(2s)}, \\[4pt]
\label{r:l22}
l_{2,2}^{(2s+1)} &= -8 l_{3,3}^{(2s)} - 4 l_{4,3}^{(2s)} - l_{5,3}^{(2s)} - 10 l_{1,1}^{(2s)} - 5 l_{2,1}^{(2s)} - l_{3,1}^{(2s)}, \\[4pt]
\label{r:l32}
l_{3,2}^{(2s+1)} &= -28 l_{3,3}^{(2s)} - 32 l_{4,3}^{(2s)} - 18 l_{5,3}^{(2s)} - 6 l_{6,3}^{(2s)} - l_{7,3}^{(2s)} \nonumber \\
&\quad - 32 l_{1,1}^{(2s)} - 38 l_{2,1}^{(2s)} - 22 l_{3,1}^{(2s)} - 7 l_{4,1}^{(2s)} - l_{5,1}^{(2s)}, \\[4pt]
\label{r:l20}
l_{2,0}^{(2s+1)} &= 64 l_{3,3}^{(2s)} + 28 l_{4,3}^{(2s)} + 7 l_{5,3}^{(2s)} + 84 l_{1,1}^{(2s)} + 38 l_{2,1}^{(2s)} + 8 l_{3,1}^{(2s)}.
\end{align}
Rewriting $\pi\bigl(D^{(2s+1)}(l_{1,2}, l_{2,0})\bigr)$ using \eqref{r:l12}--\eqref{r:l20} and extracting terms that are less than $2s+2$ by \eqref{pi-even} and \eqref{pi-even-1133}, we obtain
\[
D^{(2s+1)}(l_{1,2}, l_{2,0}) \equiv D^{(2s)}(l_{3,3} + l_{1,1},\; l_{5,3} + 2l_{2,1}) \pmod{2^{2s+2}}.
\]
From \eqref{r:11}--\eqref{r:33} and
\begin{align*}
l_{5,3}^{(2s)} &= -208 l_{1,0}^{(2s-1)} - 180 l_{2,0}^{(2s-1)} - 96 l_{3,0}^{(2s-1)} - 34 l_{4,0}^{(2s-1)} - 8 l_{5,0}^{(2s-1)} - l_{6,0}^{(2s-1)} \\
&\quad - 2748 l_{1,2}^{(2s-1)} - 1944 l_{2,2}^{(2s-1)} - 926 l_{3,2}^{(2s-1)} - 306 l_{4,2}^{(2s-1)} - 68 l_{5,2}^{(2s-1)} - 8 l_{6,2}^{(2s-1)},
\end{align*}
which is again derived by Lemma \ref{U2-j-jl2}, we obtain that $\pi\bigl(D^{(2s)}(l_{3,3} + l_{1,1},\; l_{5,3} + 2l_{2,1})\bigr) \geq 2s+2$ by $\pi\bigl(D^{(2s-1)}(l_{1,2}, 2l_{2,2} + l_{3,2})\bigr) \geq 2s$ and $\pi\bigl(D^{(2s-1)}(l_{1,2}, l_{2,0})\bigr) \geq 2s$. Thus, $\pi\bigl(D^{(2s+1)}(l_{1,2}, l_{2,0})\bigr) \geq 2s+2$, which means that \eqref{pi-odd-1220} holds for $\alpha = s+1$. Similarly, rewriting $\pi\bigl(D^{(2s+1)}(l_{1,2}, 2l_{2,2} + l_{3,2})\bigr)$ using \eqref{r:l12}--\eqref{r:l32}, we find that each term is at least $2s+2$ by \eqref{pi-even}, so that $\pi\bigl(D^{(2s+1)}(l_{1,2}, 2l_{2,2} + l_{3,2})\bigr) \geq 2s+2$. Thus, we prove that \eqref{pi-odd-1220} and \eqref{pi-odd-123} hold for $\alpha = s+1$.
\end{proof}

\begin{theorem}\label{main-thm}
For each $\alpha \geq 1$ there exists an integral constant $x_{\alpha}$ such that
\begin{align*}
L_{2\alpha+2} &\equiv x_{2\alpha} L_{2\alpha} \pmod{2^{2\alpha}},\\
L_{2\alpha+1} &\equiv x_{2\alpha-1} L_{2\alpha-1} \pmod{2^{2\alpha-2}}.
\end{align*}
\end{theorem}

\begin{proof}
By Lemma \ref{pi-L}, we have
\begin{align*}
\pi\bigl(D^{(2\alpha-1)}(l_{m,i}, l_{n,j})\bigr) &\geq 2\alpha - 2,\\
\pi\bigl(D^{(2\alpha)}(l_{m,i}, l_{n,j})\bigr) &\geq 2\alpha.
\end{align*}

According to $U^{(1)}(y_0 t) \equiv p_1 t \pmod{2}$ and $U^{(0)}(p_1 t) \equiv y_0 t \pmod{2}$, we have
\begin{align}\label{eq-u}
U^{(1)}(U^{(0)}(p_1 t)) \equiv p_1 t \pmod{2}, \qquad
U^{(0)}(U^{(1)}(y_0 t)) \equiv y_0 t \pmod{2}.
\end{align}
Using \eqref{r:L1id} and \eqref{eq-L2}, we obtain $L_1 \equiv l_{1,2}^{(1)} y_0 t \equiv y_0 t \pmod{2}$ and $L_2 \equiv l_{1,1}^{(2)} p_1 t \equiv p_1 t \pmod{2}$, where $l_{1,2}^{(1)} = 1$ and $2 \nmid l_{1,1}^{(2)}$.

From \eqref{eq-u}, for $\alpha \geq 1$, we have
\begin{align*}
L_{2\alpha-1} \equiv l_{1,2}^{(2\alpha-1)} y_0 t \equiv y_0 t \pmod{2}, \qquad
L_{2\alpha} \equiv l_{1,1}^{(2\alpha)} p_1 t \equiv p_1 t \pmod{2}.
\end{align*}
This implies that $2 \nmid l_{1,1}^{(2\alpha)}$ and $2 \nmid l_{1,2}^{(2\alpha-1)}$. Let $x_{2\alpha-1}$ be a solution of $l_{1,2}^{(2\alpha-1)} \equiv x_{2\alpha-1} l_{1,2}^{(2\alpha+1)} \pmod{2^{2\alpha-2}}$. Then for $(n,i) \neq (1,2)$,
\begin{align*}
l_{n,i}^{(2\alpha-1)} l_{1,2}^{(2\alpha-1)} \equiv l_{n,i}^{(2\alpha-1)} x_{2\alpha-1} l_{1,2}^{(2\alpha+1)} \equiv x_{2\alpha-1} l_{1,2}^{(2\alpha-1)} l_{n,i}^{(2\alpha+1)} \pmod{2^{2\alpha-2}}.
\end{align*}
Cancelling $l_{1,2}^{(2\alpha-1)}$, we obtain $l_{n,i}^{(2\alpha-1)} \equiv x_{2\alpha-1} l_{n,i}^{(2\alpha+1)} \pmod{2^{2\alpha-2}}$. Similarly, let $x_{2\alpha}$ be a solution of $l_{1,1}^{(2\alpha)} \equiv x_{2\alpha} l_{1,1}^{(2\alpha+2)} \pmod{2^{2\alpha}}$. Then for $(n,i) \neq (1,1)$,
\begin{align*}
l_{n,i}^{(2\alpha)} l_{1,1}^{(2\alpha)} \equiv l_{n,i}^{(2\alpha)} x_{2\alpha} l_{1,1}^{(2\alpha+2)} \equiv x_{2\alpha} l_{1,1}^{(2\alpha)} l_{n,i}^{(2\alpha+2)} \pmod{2^{2\alpha}}.
\end{align*}
Cancelling $l_{1,1}^{(2\alpha)}$, we obtain $l_{n,i}^{(2\alpha)} \equiv x_{2\alpha} l_{n,i}^{(2\alpha+2)} \pmod{2^{2\alpha}}$. Thus we complete the proof of Theorem \ref{main-thm}.
\end{proof}

Based on Lemma \ref{lem-L-alpha} and Theorem \ref{main-thm}, we can directly obtain the theorem below.

\begin{theorem}\label{thm}
For $\alpha \geq 1$ and $n \geq 0$, we have
\begin{align*}
\Delta_3(\lambda_{\alpha}) \, \Delta_3(2^{\alpha+2}n + \lambda_{\alpha+2})
\equiv \Delta_3(\lambda_{\alpha+2}) \, \Delta_3(2^{\alpha}n + \lambda_{\alpha}) \pmod{2^{2\lfloor \alpha/2 \rfloor}},
\end{align*}
where
\begin{align*}
\lambda_{\alpha} = \left\{\begin{array}{ll}
            \dfrac{2^{\alpha+1} + 1}{3}, & \text{if } \alpha \text{ is even}, \\[6pt]
            \dfrac{2^{\alpha} + 1}{3}, & \text{if } \alpha \text{ is odd}.
        \end{array}\right.
\end{align*}
\end{theorem}

\section*{Appendix. The fundamental relations}
Group I:
\begin{align*}
U^{(0)}(t)=&p_0(2^4t^3+3\cdot2^3t^2+11\cdot2t+7)+y_0(-2t-1);\\
U^{(0)}(t^2)=&p_0(2^6t^5+2^7t^4+9\cdot2^4t^3+19\cdot2^2t^2+7\cdot2t)+y_0(-2^3t^3-2^3t^2-2t).
\end{align*}

Group II:
\begin{align*}
U^{(0)}(p_1t)=&p_0(2^8t^7+5\cdot2^7t^6+7\cdot2^7t^5+45\cdot2^4t^4+21\cdot2^4t^3+21\cdot2^2t^2+2^3t)\\
&+y_0(-2^5t^5-3\cdot2^4t^4-2^5t^3-5\cdot2t^2-t);\\
U^{(0)}(p_1t^2)=&p_0(2^{10}t^9+3\cdot2^{10}t^8+39\cdot2^7t^7+39\cdot2^7t^6+201\cdot2^4t^5+167\cdot2^3t^4+83\cdot2^2t^3\\
&+19\cdot2t^2)+y_0(-2^7t^7-2^8t^6-15\cdot2^4t^5-2^7t^4-19\cdot2t^3-5t^2).
\end{align*}

Group III:
\begin{align*}
U^{(1)}(y_0t)=&p_1(2^{13}t^9+3\cdot2^{13}t^8+69\cdot2^9t^7+125\cdot2^8t^6+617\cdot2^5t^5+131\cdot2^6t^4+597\cdot2^2t^3\\
&+105\cdot2^2t^2+35t)-2^{13}t^{11}-7\cdot2^{12}t^{10}-81\cdot2^9t^9-75\cdot2^9t^8-737\cdot2^5t^7\\
&-621\cdot2^4t^6-687\cdot2^2t^5-225\cdot2t^4-9\cdot2^2t^3;\\
U^{(1)}(y_0t^2)=&p_1(2^{15}t^{11}+7\cdot2^{14}t^{10}+95\cdot2^{11}t^9+103\cdot2^{11}t^8+1251\cdot2^7t^7+1371\cdot2^6t^6+543\cdot2^6t^5\\
&+1211\cdot2^3t^4+443\cdot2^2t^3+91\cdot2t^2+7t)-2^{15}t^{13}-2^{17}t^{12}-111\cdot2^{11}t^{11}-245\cdot2^{10}t^{10}\\
&-1495\cdot2^7t^9-817\cdot2^7t^8-5\cdot2^{13}t^7-1387\cdot2^3t^6-243\cdot2^3t^5-93\cdot2t^4-2^3t^3.\\
\end{align*}

Group IV:
\begin{align*}
U^{(1)}(p_0t)=&p_1(2^{8}t^7+5\cdot2^{7}t^6+11\cdot2^6t^5+7\cdot2^6t^4+11\cdot2^4t^3+5\cdot2^3t^2+2^2t)\\
&-2^8t^9-3\cdot2^8t^8-13\cdot2^6t^7-17\cdot2^5t^6-13\cdot2^4t^5-11\cdot2^2t^4-2^2t^3;\\
U^{(1)}(p_0t^2)=&p_1(2^{10}t^9+3\cdot2^{10}t^8+33\cdot2^7t^7+55\cdot2^6t^6+121\cdot2^4t^5+11\cdot2^6t^4+5\cdot2^5t^3\\
&+5\cdot2^2t^2+t)-2^{10}t^{11}-7\cdot2^9t^{10}-39\cdot2^7t^9-33\cdot2^7t^8-145\cdot2^4t^7\\
&-103\cdot2^3t^6-45\cdot2^2t^5-5\cdot2^2t^4-t^3.
\end{align*}

\subsection*{Acknowledgements}
 The second author was supported by  the National Natural Science Foundation of China (Grant No. 12401438).







\begin{thebibliography}{10}
\bibitem{Andrews-1976}
G.~E. Andrews, {\it The theory of partitions}, Encyclopedia of Mathematics and its Applications, Vol. 2, Addison-Wesley Publishing Co., Reading, Mass.-London-Amsterdam, 1976; \MR{0557013}
\bibitem{Andrews-Paule-2007}
G.~E. Andrews and  P. Paule, MacMahon's partition analysis. XI. Broken diamonds and modular forms. Acta Arith. 126 (2007), no. 3, 281--294; \MR{2289961}
\bibitem{Atkin-1967}
A.~O.~L. Atkin, Proof of a conjecture of Ramanujan, Glasgow Math. J. {\bf 8} (1967), 14--32; \MR{0205958}
\bibitem{Atkin-Lehner-1970}
A.~O.~L. Atkin and J. Lehner, Hecke operators on $\Gamma \sb{0}(m)$, Math. Ann. {\bf 185} (1970), 134--160; \MR{0268123}
\bibitem{Chan-2008}
S.~H. Chan, Some congruences for Andrews-Paule's broken 2-diamond partitions, Discrete Math. {\bf 308} (2008), no.~23, 5735--5741; \MR{2459393}
\bibitem{Garvan-1984}
F.~G. Garvan, A simple proof of Watson's partition congruences for powers of $7$, J. Austral. Math. Soc. Ser. A {\bf 36} (1984), no.~3, 316--334; \MR{0733905}
\bibitem{gtutorial}
F.~G. Garvan, A tutorial for the MAPLE ETA package, arXiv:1907.09130;
\bibitem{Hirschhorn-1981}
M.~D. Hirschhorn and D.~C. Hunt, A simple proof of the Ramanujan conjecture for powers of $5$, J. Reine Angew. Math. {\bf 326} (1981), 1--17; \MR{0622342}
\bibitem{Hirschhorn-Sellers-2007}
M.~D. Hirschhorn and J.~A. Sellers, On recent congruence results of Andrews and Paule for broken $k$-diamonds, Bull. Austral. Math. Soc. {\bf 75} (2007), no.~1, 121--126; \MR{2309551}
\bibitem{Paule-Radu-2010}
P. Paule and C.-S. Radu, Infinite families of strange partition congruences for broken 2-diamonds, Ramanujan J. {\bf 23} (2010), no.~1-3, 409--416; \MR{2739226}
\bibitem{Paule-Radu-2012}
P. Paule and C.-S. Radu, The Andrews-Sellers family of partition congruences, Adv. Math. {\bf 230} (2012), no.~3, 819--838; \MR{2921161}
\bibitem{Radu-Sellers-2013}
S. Radu, J.~A. Sellers, An extensive analysis of the parity of broken 3-diamond partitions, J. Number Theory{\bf 133} (2013), no.~11, 3703--3716; \MR{3084296}
\bibitem{Ramanujan-1919}
S. Ramanujan, Some properties of $p(n)$, the number of partitions of $n$, Proc. Camb. Philos. Soc. 19, 214-216 (1919). \MR{2280868}
\bibitem{Xia-2014}
E.~X.~W. Xia, New congruences modulo powers of 2 for broken 3-diamond partitions and 7-core partitions, J. Number Theory {\bf 141} (2014), 119--135; \MR{3195392}
\bibitem{Xiong-2011}
X. Xiong, Two congruences involving Andrews-Paule's broken 3-diamond partitions and 5-diamond partitions, Proc. Japan Acad. Ser. A Math. Sci. {\bf 87} (2011), no.~5, 65--68; \MR{2803892}
\bibitem{Watson-1938}
G.~N. Watson, Ramanujans Vermutung \"uber Zerf\"allungszahlen, J. Reine Angew. Math. {\bf 179} (1938), 97--128; \MR{1581588}



%
%
%





%











%
%
%
%



\end{thebibliography}
\end{document}